\documentclass[11pt]{article}
\usepackage{graphicx}
\usepackage{amsmath,amssymb}

\usepackage{graphicx}
\usepackage{amssymb}
\usepackage{amsmath}
\newtheorem{lemma}{Lemma}
\newtheorem{myclaim}{Claim}
\newtheorem{definition}{Definition}
\newtheorem{theorem}{Theorem}

\renewcommand{\phi}{\varphi}
\newenvironment{proof}{\noindent{\sf Proof.}}{\hfill $\boxtimes\hspace{2mm}$\linebreak}
\newcommand{\qed}{\hfill $\boxtimes\hspace{1mm}$}
\newenvironment{proof-of-claim}{\noindent{\em Proof of Claim.}}{\hfill $\boxtimes\hspace{2mm}$\linebreak}

\begin{document}

\title{Insert your title here
}


\author{Sanaz Azimipour         \and
        Pavel Naumov 
}




\title{Axiomatic Theory of Betweenness}


\maketitle 

\begin{abstract}
Betweenness as a relation between three individual points has been widely studied in geometry and axiomatized by several authors in different contexts. The article proposes a more general notion of betweenness as a relation between three sets of points. The main technical result is a sound and complete logical system describing universal properties of this relation between sets of vertices of a graph.
\end{abstract}






\section{Introduction}

In this article we develop an axiomatic theory of the betweenness relation. Such a relation could be considered as a relation between points or a relation between sets of points. 

\subsection{Betweenness of Points}

Betweenness of points is a commonly studied relation in geometry. Usually it has been investigated not as a stand-alone notion, but in the context of comprehensive axiomatic theories of geometry. For example, Hilbert's axiomatisation of Euclidean geometry~\cite{h02} treats the relation ``between" as a primitive (non-definable) relation between three points. Three of his ``order" axioms are concerned with this relation:
\begin{enumerate}
    \item If a point $b$ lies between points $a$ and $c$, then $b$ is also between $c$ and $a$, and there exists a line containing the distinct points $a$,$b$, and $c$.
    \item If $a$ and $c$ are two points, then there exists at least one point $b$ on the line $ac$ such that $b$ lies between $a$ and $c$ and at least one point $d$ so situated that $c$ lies between $a$ and $d$.
    \item Of any three points situated on a straight line, there is always one and only one which lies between the other two.
\end{enumerate}
Note that although betweenness is a relation between points and not between lines, these axioms refer to all three primitive terms of Hilbert's axiomatization: betweenness, points, and lines. Furthermore, some properties of betweenness might not be captured by these three axioms at all and, instead, they might follow from the combination of these axioms and other Hilbert's axioms.

\begin{figure}[ht]
\begin{center}
\vspace{0mm}
\scalebox{.6}{\includegraphics{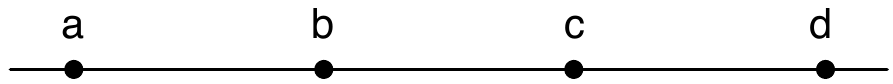}}
\vspace{0mm}
\footnotesize
\caption{If point $b$ is between points $a$ and $c$ and point $c$ is between points $b$ and $d$, then point $b$ is between points $a$ and $d$.}\label{huntington kline axiom figure}
\vspace{0cm}
\end{center}
\vspace{0mm}
\end{figure}

Huntington and Kline~\cite{hk17tams} proposed several systems of axioms for betweenness of points on a line. Their axioms are self-contained in the sense that they do not refer to any other primitive terms. The example of an axiom in one of their systems, see Figure~\ref{huntington kline axiom figure}, is ``If point $b$ is between points $a$ and $c$ and point $c$ is between points $b$ and $d$, then point $b$ is between points $a$ and $d$".

Betweenness as a relation between three points could be generalised from a relation between points on a line to a relation between points on a plane by saying that a point $b$ is between points $a$ and $c$ if point $b$ belongs to the open interval with the end points $a$ and $c$. This could be even further generalised to a relation between points in a metric space through the triangle inequality. Namely, we can say that $b$ is between $a$ and $c$ if $b$ is not equal to either of these two points and $d(a,c)=d(a,b)+d(b,c)$. 

Another way to generalise betweenness is to consider this relation between vertices on a graph. We can say that vertex $b$ is between vertices $a$ and $c$ if $b$ is an internal vertex of each path from vertex $a$ to vertex $c$, see Figure~\ref{every path figure}. This relation has a close connection to the betweenness on road systems recently studied by Bankston~\cite{b13bms}.

\begin{figure}[ht]
\begin{center}
\vspace{0mm}
\scalebox{.6}{\includegraphics{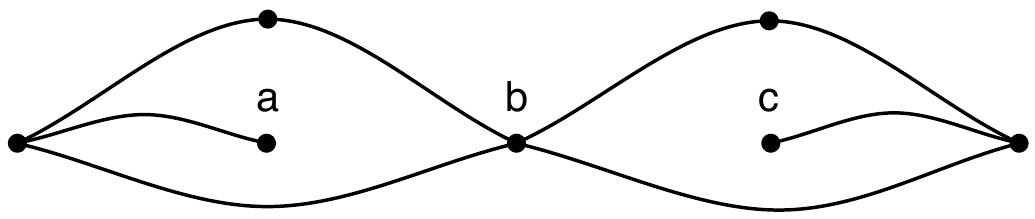}}
\vspace{0mm}
\footnotesize
\caption{Vertex $b$ is between vertices $a$ and $c$ if $b$ is an internal vertex of each path from vertex $a$ to vertex $c$.}\label{every path figure}
\vspace{0cm}
\end{center}
\vspace{0mm}
\end{figure}

Finally, it is also possible to consider betweenness as a relation on partial orders and other similar structures. See Fishburn~\cite{f71jpaa} for a review of the results in this area.

\subsection{Betweenness of Sets}

Betweenness could be also considered as a relation between {\em sets} of points. For any three sets $A,B,C\subseteq\mathbb{R}$, we say that the set $B$ is between sets $A$ and $C$ if for any $a\in A$ and any $c\in C$ there is a point $b\in B$ such that point $b$ is between points $a$ and $c$. We denote this relation between sets $A$, $B$, and $C$ by $A|B|C$. For example, $\mathbb{Q}|\mathbb{Q}|\mathbb{R}\setminus\mathbb{Q}$. In other words, the set of all rational numbers $\mathbb{Q}$ is between itself and the set of all irrational numbers. This statement is true because every open interval contains at least one rational point. There are at least three natural generalisations of this relation.

First, for any sets $A,B,C\subseteq\mathbb{R}^2$ we can say that $A|B|C$ if for any $a\in A$ and any $c\in C$ there is a point $b\in B$ such that point $b$ is an internal point of the interval with end points $a$ and $c$. This notion of set betweenness could be generalised to sets in an arbitrary metric space if the ``point of the interval" requirement is replaced with $d(a,c)=d(a,b)+d(b,c)$.

\begin{figure}[ht]
\begin{center}
\vspace{0mm}
\scalebox{.6}{\includegraphics{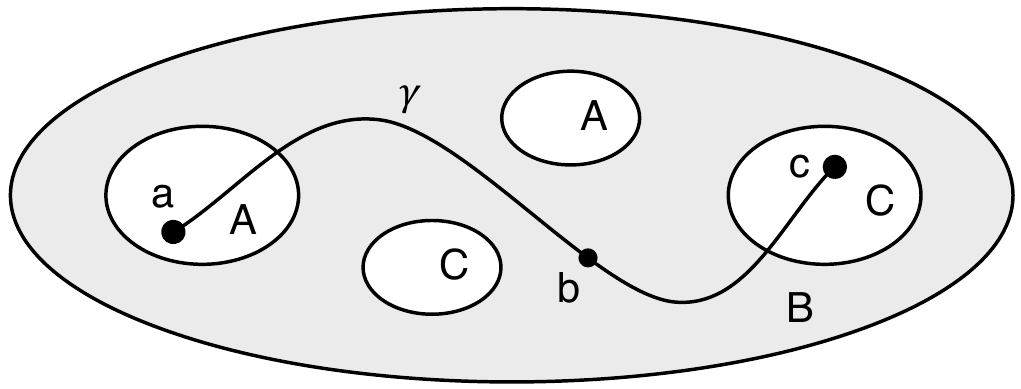}}
\vspace{0mm}
\footnotesize
\caption{$A|B|C$ if for any $a\in A$, any $c\in C$ and any curve $\gamma$ from point $a$ to point $b$, there is $b\in B$ such that $b$ is an internal point of curve $\gamma$.}\label{curve example figure}
\vspace{0cm}
\end{center}
\vspace{0mm}
\end{figure}

Second, for any sets $A,B,C\subseteq\mathbb{R}^2$ we can say that $A|B|C$ if for any $a\in A$, any $c\in C$ and any curve $\gamma$ from point $a$ to point $b$, there is $b\in B$ such that $b$ is an internal point of curve $\gamma$, see Figure~\ref{curve example figure}. This notion of betweenness could be generalised to a relation between sets of points in an arbitrary topological space.

Finally, see Figure~\ref{graph figure}, we can consider set betweenness on graphs. For any sets of vertices $A$, $B$, and $C$, we say that the set $B$ is between sets $A$ and $C$ if for each vertex $a\in A$, each vertex $c\in C$, and each path from vertex $a$ to vertex $c$ there is an internal vertex of this path that belongs to the set $B$. This notion of betweenness, mostly between edges rather than graphs, has been used by the second author to describe information flow properties in communication networks~\cite{mn11apal,mn11amai}.

\begin{figure}[ht]
\begin{center}
\vspace{0mm}
\scalebox{.6}{\includegraphics{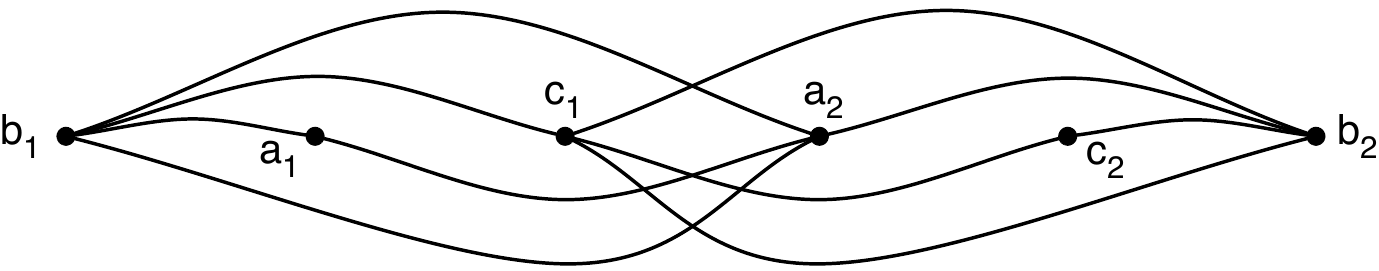}}
\vspace{0mm}
\footnotesize
\caption{$\{a_1,a_2\}|\{b_1,b_2\}|\{c_1,c_2\}$,  because for each vertex $a\in \{a_1,a_2\}$, each vertex $c\in \{c_1,c_2\}$ and each path from vertex $a$ to vertex $c$ there is an internal vertex of this path that belongs to the set $\{b_1,b_2\}$.}\label{graph figure}
\vspace{0cm}
\end{center}
\vspace{0mm}
\end{figure}

\subsection{Insertion Principle}

One of the more interesting observations about betweenness is that if point $b$ is between points $a$ and $c$, and point $i$ is between points $a$ and $b$, then point $i$ is between points $a$ and $c$. We call this statement the ``insertion principle", because it  can be informally rephrased as ``if a point $b$ is between points $a$ and $c$ and  a point $i$ is {\em inserted} between points $a$ and $b$, then point $i$ is also between points $a$ and $c$", see Figure~\ref{insertion linear}. Using our notation for betweenness, this principle can be written as $\{a\}|\{b\}|\{c\}\to (\{a\}|\{i\}|\{b\}\to \{a\}|\{i\}|\{c\})$, or, omitting curly braces, $a|b|c\to (a|i|b\to a|i|c)$.

\begin{figure}[ht]
\begin{center}
\vspace{0mm}
\scalebox{.6}{\includegraphics{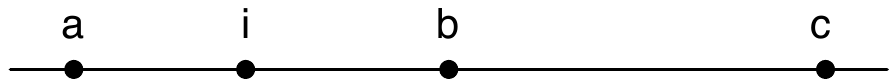}}
\vspace{0mm}
\footnotesize
\caption{Insertion Principle: If point $b$ is between points $a$ and $c$ and a point $i$ is {\em inserted} between points $a$ and $b$, then point $i$ is also between points $a$ and $c$.}\label{insertion linear}
\vspace{0cm}
\end{center}
\vspace{0mm}
\end{figure}

The insertion principle is a very general property of betweenness. For example, see Figure~\ref{insertion curve}, it is true for any sets of points on a plane.

\begin{figure}[ht]
\begin{center}
\vspace{0mm}
\scalebox{.6}{\includegraphics{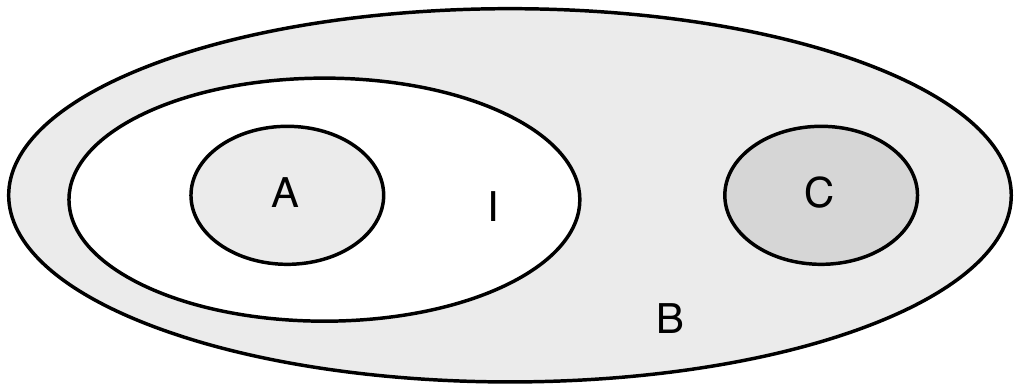}}
\vspace{0mm}
\footnotesize
\caption{$A|B|C\to(A|I|B\to A|I|C)$.}\label{insertion curve}
\vspace{0cm}
\end{center}
\vspace{0mm}
\end{figure}

The betweenness statement $A|B|C$ is equivalent to $C|B|A$ and, thus, there is a symmetry between the first and the third argument of the betweenness predicate. The insertion principle, as stated so far, is not symmetric with respect to these two arguments. As a result, a valid symmetric version of this principle can be stated: $A|B|C\to(B|I|C\to A|I|C)$. The original principle ``inserts" a set $I$ between sets $A$ and $B$, when as the second insertion principle ``inserts" a set $I$ between sets $B$ and $C$. What is more interesting is that there is an even more general form of the insertion principle: 
\begin{equation}\label{symmetric insertion}
    A|B_1,B_2|C\to(A|I|B_1\to(B_2|I|C\to A|I|C)),
\end{equation}
where $B_1,B_2$ denotes the union of sets $B_1$ and $B_2$. This principle is illustrated in Figure~\ref{insertion curve second}. Informally, this principle splits set $B$ into parts $B_1$ and $B_2$ and inserts a set $I$ between $A$ and $B_1$ and between $B_2$ and $C$.

\begin{figure}[ht]
\begin{center}
\vspace{0mm}
\scalebox{.6}{\includegraphics{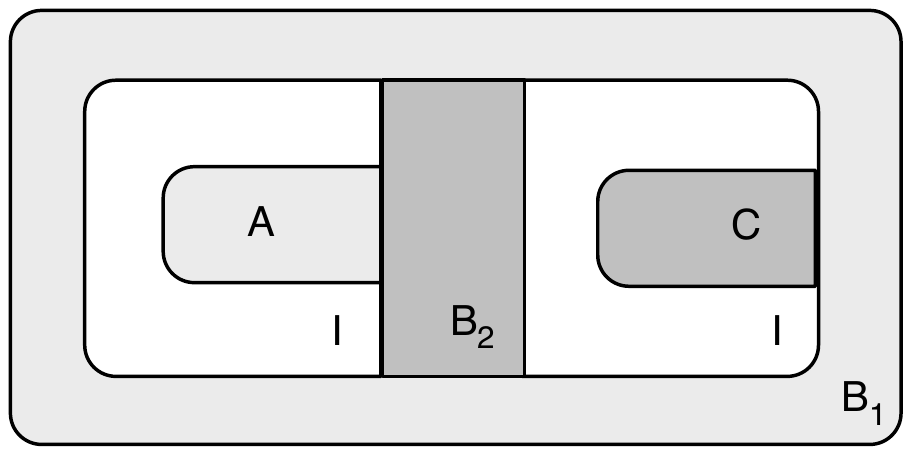}}
\vspace{0mm}
\footnotesize
\caption{$A|B_1,B_2|C\to(A|I|B_1\to(B_2|I|C\to A|I|C))$.}\label{insertion curve second}
\vspace{0cm}
\end{center}
\vspace{0mm}
\end{figure}

It is relatively easy to see why principle~(\ref{symmetric insertion}) is true. Indeed, consider any curve from a point in the set $A$ to a point in the set $C$. By the first assumption, this curve must have an internal point either belongs to set $B_1$ or to set $B_2$. Without loss of generality, assume that the curve contains an internal point from set $B_1$. Therefore, the curve must also contain an internal point from the set $I$ due to the second assumption of the formula~(\ref{symmetric insertion}).

\begin{figure}[ht]
\begin{center}
\vspace{0mm}
\scalebox{.6}{\includegraphics{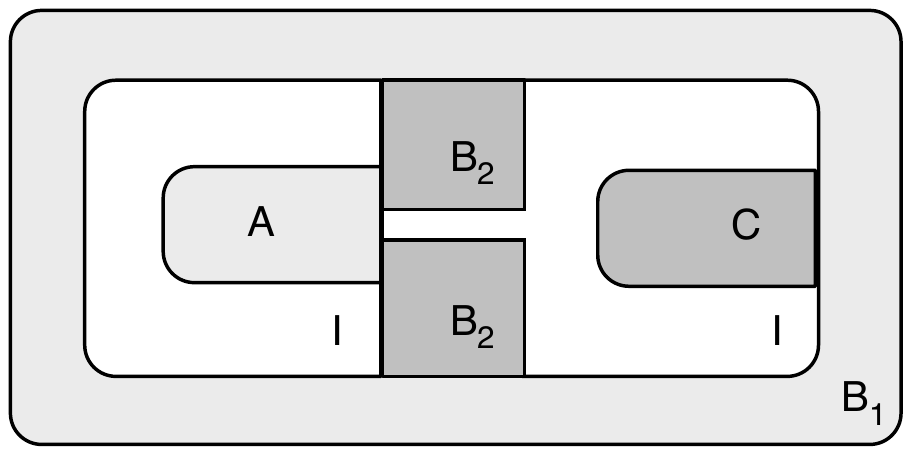}}
\vspace{0mm}
\footnotesize
\caption{$A|B_1,I,B_2|C\to(A|I|B_1\to(B_2|I|C\to A|I|C))$.}\label{insertion curve third}
\vspace{0cm}
\end{center}
\vspace{0mm}
\end{figure}

The above argument could be easily modified to prove an even stronger version of principle~(\ref{symmetric insertion}). Namely,
\begin{equation}\label{symmetric insertion stronger}
    A|B_1,I,B_2|C\to(A|I|B_1\to(B_2|I|C\to A|I|C)).
\end{equation}
This principle is illustrated in Figure~\ref{insertion curve third}. In this article we give a partial answer to the question, what is the strongest form of the insertion principle. It turns out the answer to this question depends on the setting in which the betweenness relation is considered. The main focus of our work is on betweenness as a relation on sets of vertices of a graph. In this setting, principle~(\ref{symmetric insertion stronger}) has an even stronger form:
\begin{equation}\label{symmetric insertion strongest}
    A|B_1,I,B_2|C\to(A|I,C|B_1\to(B_2|A,I|C\to A|I|C)).
\end{equation}
We prove this form of the insertion principle for the finite graph semantics in Lemma~\ref{insertion axiom soundness}. Informally, the main technical result of this article is that (\ref{symmetric insertion strongest}) is the strongest possible form of the insertion principle for graphs. More formally, we prove that the logical system consisting of axiom~(\ref{symmetric insertion strongest}) and several other much more straightforward properties of betweenness is sound and complete with respect to the graph semantics.  

Unlike principle~(\ref{symmetric insertion stronger}), insertion principle~(\ref{symmetric insertion strongest}) is not valid for arbitrary sets of points on a plane. It is valid, however, if sets $A$, $B$, $C$, \dots are arbitrary {\em closed} sets on a plane, or, more generally, arbitrary closed sets in a topological space. Furthermore, since finite graphs can be embedded into $\mathbb{R}^3$, it is likely that our proof of completeness for graphs could be modified to obtain the completeness of our logical system with respect to closed sets in $\mathbb{R}^3$.    

\subsection{Outline}

The article is organised as follows. In the next section we formally define the language of our logical system. In Section~\ref{semantics section}, we introduce graph semantics for this language. In Section~\ref{axioms section}, we list axioms of our formal system. We prove soundness of these axioms in Section~\ref{soundness section} and completeness of our logical system in Section~\ref{completeness section}. Section~\ref{conclusion section} concludes the article by discussing non-strict betweenness on graphs and showing that insertion principle in the form~(\ref{symmetric insertion strongest}) is not, generally speaking, valid for sets of points on a plane.

\section{Syntax}\label{syntax section}

In this section we introduce the syntax of our formal theory of betweenness. Informally, the language of our theory includes betweenness statements of the form $A|B|C$ and all possible Boolean combinations of these statements. This is a propositional theory in the sense that we do not allow the use of quantifiers. Since all Boolean connectives can be expressed through negation and implication, we use only these two in our formal syntax.

\begin{definition}\label{Phi definition}
For any finite set $V$ of ``vertices", let the language $\Phi(V)$ be the minimal set of formulae such that
\begin{enumerate}
    \item $A|B|C\in\Phi(V)$ for all sets $A,B,C\subseteq V$,
    \item $(\phi\to\psi)\in\Phi(V)$ for all $\phi,\psi\in\Phi(V)$,
    \item $\neg\phi\in\Phi(V)$ for each $\phi\in\Phi(V)$.
\end{enumerate}
\end{definition}
For the sake of simplicity, when listing elements of sets $A$, $B$, and $C$ explicitly, we usually omit curly brackets in the expression $A|B|C$. For example, we write $a|b_1,b_2|c$ instead of $\{a\}|\{b_1,b_2\}|\{c\}$.

\section{Semantics}\label{semantics section}

In this article by graph we mean an undirected graph without multiple edges, but possibly with loops. Minor changes are needed to accommodate graphs with multiple edges or to exclude graphs with loops. It is likely that our results can be adopted to directed graphs, but this would require a more substantial revision.

\begin{definition}
A path between a vertex $a$ and a vertex $b$ in a graph $(V,E)$ is any sequence of vertices $a=v_0,v_1,\dots,v_n=b$, where $n\ge 0$, such that  $(v_{i},v_{i+1})\in E$ for each $0\le i< n$. Vertices $v_1,\dots,v_{n-1}$ are called internal vertices of the path.
\end{definition}

Next is the key definition of this article. Its first item formally specifies the betweenness relation as a relation between sets of vertices of a graph. 
\begin{definition}\label{sat-strict}
For any $\phi\in\Phi(V)$ and any graph $(V,E)$, satisfiability relation $(V,E)\vDash \phi$ is defined inductively as follows
\begin{enumerate}
    \item $(V,E)\vDash A|B|C$, if for any $a\in A$, any $c\in C$, and any path between vertices $a$ and $c$, at least one internal vertex of the path belongs to the set $B$.
    \item $(V,E)\vDash\neg\phi$, if $(V,E)\nvDash \phi$,
    \item $(V,E)\vDash\phi\to\psi$, if $(V,E)\nvDash\phi$ or $(V,E)\vDash\psi$.
\end{enumerate}
\end{definition}

Note that item 1 of the above definition requires that at least one {\em internal} vertex of the path belongs to set $B$. If the requirement of the vertex to be internal is removed, then we get the definition of what we could call {\em non-strict} betweenness relation. 

\section{Axioms}\label{axioms section}

For any given finite set $V$, our axiomatic system consists of the following axioms in the language $\Phi(V)$:
\begin{enumerate}
    \item Trivial Path: $\neg(A|B|C)$, if $A\cap C\neq\varnothing$,
    \item Empty Set: $\varnothing|B|C$,
    \item Shortest Path: $A|B|C\to A|(B\!\setminus\!A)|C$,
    \item Aggregation: $A_1|B|C \to(A_2|B|C\to A_1,A_2|B|C)$,
    \item Symmetry: $A|B|C\to C|B|A$,
    \item Left Monotonicity: $A_1,A_2 |B|C\to A_1|B|C$,
    \item Central Monotonicity: $A|B_1|C\to A|B_1,B_2|C$,
    \item Insertion: $A|B_1,I,B_2|C\to(A|I,C|B_1\to(B_2|A,I|C\to A|I|C))$,
    \item Transitivity: $\neg(A|B|d)\to(\neg(d|B|C)\to\neg(A|B|C))$, if $d\notin B$.
\end{enumerate}
The name Shortest Path comes from the shortest path used in the proof of the soundness of this axiom, see Lemma~\ref{shortest path sound}. In the above axioms by $A,B$ we denote the union of sets $A$ and $B$. Note that we represent union by comma only inside betweenness predicate. In all other setting, to avoid confusion, we use the standard notations $A\cup B$. 

We write $\vdash_V\phi$ if formula $\phi$ is provable from the the set of all propositional tautologies and the above axioms using Modus Ponens inference rule. We write $X\vdash_V\phi$, if a formula $\phi$ is derivable with the use of additional axioms from the set $X$. We often omit the subscript $V$ when its value is clear from the context.

\section{Soundness}\label{soundness section}

In this section we prove the soundness of our logical system. We prove soundness of each axiom as a separate lemma for an arbitrary graph $(V,E)$, arbitrary sets $A,B,C,A_1,A_2,B_1,B_2,I\subseteq V$, and an arbitrary vertex $d\in V$. The soundness theorem that follows from these lemmas is stated at the end of this section.


\begin{lemma}
$(V,E)\nvDash A|B|C$, if $A\cap C\neq \varnothing$. 
\end{lemma}
\begin{proof}
Fix an arbitrary $v\in A\cap C$. Consider trivial path consisting of the single vertex $v$. This path has no internal vertices. Therefore, $(V,E)\nvDash A|B|C$ by Definition~\ref{sat-strict}.
\end{proof}

\begin{lemma}
$(V,E)\vDash \varnothing|B|C$.
\end{lemma}
\begin{proof}
Due to Definition~\ref{sat-strict}, the statement of the lemma is vacuously true because the set $\varnothing$ contains no elements.
\end{proof}

\begin{lemma}\label{shortest path sound}
If $(V,E)\vDash A|B|C$, then $(V,E)\vDash A|(B\!\setminus\!A)|C$.
\end{lemma}
\begin{proof}
Suppose that $(V,E)\nvDash A|(B\!\setminus\!A)|C$. Thus, by Definition~\ref{sat-strict}, there is a path from a vertex in set $A$ to a vertex in set $C$ that has no internal vertices belonging to set $B\setminus A$. Recall that graph $(V,E)$ is finite by Definition~\ref{Phi definition}. Hence, there must exist a {\em shortest} path $\pi$ from a vertex in set $A$ to a vertex in set $C$ that has no internal vertices belonging to set $B\setminus A$. Because path $\pi$ is shortest, it cannot contain an internal vertex from set $A$. Thus, $\pi$ is a path from a vertex in set $A$ to a vertex in set $C$ that has no internal vertices belonging to set $B$. Therefore, $(V,E)\nvDash A|B|C$ by Definition~\ref{sat-strict}.    
\end{proof}

\begin{lemma}
If $(V,E)\vDash A_1|B|C$ and $(V,E)\vDash A_2|B|C$, then $(V,E)\vDash A_1,A_2|B|C$. 
\end{lemma}
\begin{proof}
Consider any $a\in A_1\cup A_2$, any $c\in C$, and any path $a=v_0,\dots,v_n=c$. Without loss of generality, we can assume that $a\in A_1$. Thus, by the assumption $(V,E)\vDash A_1|B|C$ and Definition~\ref{sat-strict}, there must exist $0<i<n$ such that $v_i\in B$. Therefore, $(V,E)\vDash A_1,A_2|B|C$ by Definition~\ref{sat-strict}.
\end{proof}

\begin{lemma}
If $(V,E)\vDash A|B|C$, then $(V,E)\vDash C|B|A$.
\end{lemma}
\begin{proof}
Consider any $c\in C$, any $a\in A$, and any path $c=v_0,\dots,v_n=a$. Since graph $(V,E)$ is not directed, the sequence $a=v_n,\dots,v_0=c$ is also a path in this graph. Thus, by the assumption $(V,E)\vDash A|B|C$ and Definition~\ref{sat-strict}, there exists $0<i<n$ such that $v_i\in B$. Therefore, $(V,E)\vDash C|B|A$ by Definition~\ref{sat-strict}.
\end{proof}

\begin{lemma}
If $(V,E)\vDash A_1,A_2|B|C$, then $(V,E)\vDash A_1|B|C$.
\end{lemma}
\begin{proof}
Consider any $a\in A_1$, any $c\in C$, and any path $a=v_0,\dots,v_n=c$. Note that $a\in A_1\subseteq A_1\cup A_2$. Thus, by the assumption $(V,E)\vDash A_1,A_2|B|C$ and Definition~\ref{sat-strict}, there exists $0<i<n$ such that $v_i\in B$. Therefore, $(V,E)\vDash A_1|B|C$ by Definition~\ref{sat-strict}.
\end{proof}

\begin{lemma}
If $(V,E)\vDash A|B_1|C$, then $(V,E)\vDash A|B_1,B_2|C$.
\end{lemma}
\begin{proof}
Consider any $a\in A$, any $c\in C$, and any path $a=v_0,\dots,v_n=c$.  By the assumption $(V,E)\vDash A|B_1|C$ and Definition~\ref{sat-strict}, there exists $0<i<n$ such that $v_i\in B_1$. Thus, $v_i\in B_1\cup B_2$. Therefore, $(V,E)\vDash A|B_1,B_2|C$ by Definition~\ref{sat-strict}.
\end{proof}

\begin{lemma}\label{insertion axiom soundness}
If $(V,E)\vDash A|B_1,I,B_2|C$,  $(V,E)\vDash A|I,C|B_1$, $(V,E)\vDash B_2|A,I|C$,
then $(V,E)\vDash A|I|C$. 
\end{lemma}
\begin{proof}
Consider any $a\in A$, any $c\in C$, and any path $a=v_0,\dots,v_n=c$. It suffices to prove that there is $0<i<n$ such that $v_i\in I$. Suppose the opposite, that is, $v_1,\dots,v_{n-1}\notin I$.

Note that $v_0=a\in A$. Let $k$ be the largest integer such that $v_k\in A$ and $0\le k\le n$. Thus, $v_{k+1},v_{k+2},\dots,v_{n}\notin A$. Note now that $v_n=c\in C$. Let $m$ be the smallest integer such that $v_m\in C$ and $k\le m\le n$. Hence, $v_k,\dots,v_{m-1}\notin C$. Therefore, the path $v_k,v_{k+1},\dots,v_{m-1},v_m$ is such that $v_k\in A$, $v_m\in C$, and $v_{k+1},\dots,v_{m-1}\notin A\cup I\cup C$.

By Definition~\ref{sat-strict}, the assumption $(V,E)\vDash A|B_1,I,B_2|C$ implies that there is $k<\ell<m$ such that $v_\ell\in B_1\cup I\cup B_2$. Thus, $v_\ell\in B_1\cup B_2$ because $v_\ell\notin A\cup I\cup C$. Without loss of generality, we can assume that $v_\ell\in B_1$. Hence, the path $v_k,v_{k+1},\dots,v_{\ell-1},v_\ell$ is such that $v_k\in A$ and $v_\ell\in B_1$. Then, by the assumption $(V,E)\vDash A|I,C|B_1$ and due to Definition~\ref{sat-strict}, there must exist $k<i<\ell$ such that $v_i\in I\cup C$. The last statement contradicts the established above fact that $v_{k+1},\dots,v_{m-1}\notin A\cup I\cup C$, because $k<i<\ell<m$.
\end{proof}

\begin{lemma}
If $d\notin B$, $(V,E)\nvDash A|B|d$, $(V,E)\nvDash d|B|C$,
then $(V,E)\nvDash A|B|C$.
\end{lemma}
\begin{proof}
By Definition~\ref{sat-strict}, the assumption $(V,E)\nvDash A|B|d$ implies that there is a path $v_0,v_1,\dots,v_k=d$ such that $v_0\in A$ and $v_i\notin B$ for each $0<i<k$. Similarly, the assumption $(V,E)\nvDash d|B|C$, by Definition~\ref{sat-strict}, implies that there is a path $d=u_0,u_1,\dots,u_n$ such that $u_n\in C$ and $u_i\notin B$ for all $0<i<n$. Recall that $d\notin B$ by the assumption of the lemma. Thus, sequence $v_0,v_1,\dots,v_{k-1},d,u_1,u_2,\dots,u_n$ is a path whose internal vertices do not belong to the set $B$. Additionally, $v_0\in A$ and $u_n\in C$. Therefore, $(V,E)\nvDash A|B|C$ by Definition~\ref{sat-strict}.
\end{proof}

The soundness theorem below follows from the above lemmas by induction on the length of derivation.
\begin{theorem}
For each formula $\phi\in\Phi(V)$, if $\vdash_V\phi$,
then $(V,E)\vDash\phi$ for each graph $(V,E)$. \qed
\end{theorem}

\section{Completeness}\label{completeness section}

The soundness theorem proved in the previous section states that each theorem of our system is valid in each finite graph. 
In this section we prove the converse of this statement, known as the completeness theorem. The proof of the completeness theorem consists in constructing a countermodel for each statement not provable from the axioms.

\begin{theorem}
For any formula $\phi\in \Phi(V)$, if $(V,E)\vDash \phi$ for each graph $(V,E)$, then $\vdash_V \phi$.
\end{theorem}
\begin{proof}
Suppose that $\nvdash\phi$. By Lindenbaum's lemma~\cite[Proposition 2.14]{m09}, there is a maximal consistent set of $X\subseteq\Phi(V)$ such that $\neg\phi\in X$. We need to specify a relation $E\subseteq V^2$ such that $(V,E)\nvDash\phi$. This is done in Definition~\ref{E-strict} below.

\begin{definition}\label{G-strict}
$\mathcal{G}(a,c)=\{G\subseteq V \;|\; X\vdash a|G|c \mbox{ and } a,c\notin G \}$.
\end{definition}
Informally $\mathcal{G}(a,c)$ is the family of all sets that ``separate'' vertices $a$ and $c$. We define vertices $a$ and $c$ to be adjacent if they are not ``separated'' by any set.

\begin{definition}\label{E-strict}
$E=\{(a,c)\in V^2\;|\; \mathcal{G}(a,c)=\varnothing\}$.
\end{definition}

The next six lemmas are about arbitrary sets $A,B,C\subseteq V$.

\begin{lemma}\label{edge gateways strict}
$X\nvdash a|B|c$, for each edge $(a,c)\in E$.  
\end{lemma}
\begin{proof}
Suppose that $X\vdash a|B|c$. Thus, $X\vdash a|B\setminus\{a\}|c$ by the Shortest Path axiom. Hence, $X\vdash c|B\setminus\{a\}|a$ by the Symmetry axiom. Then, $X\vdash c|B\setminus\{a,c\}|a$ again by the Shortest Path axiom. Thus, $X\vdash a|B\setminus\{a,c\}|c$ by the Symmetry axiom. Hence, $B\setminus\{a,c\}\in \mathcal{G}(a,c)$ by Definition~\ref{G-strict}. Therefore, $(a,c)\notin E$, by Definition~\ref{E-strict}. 
\end{proof}

\begin{lemma}\label{left right strict}
If $X\vdash A|B|C$, then $(V,E)\vDash A|B|C$.
\end{lemma}
\begin{proof}
Suppose $(V,E)\nvDash A|B|C$. Thus, by Definition~\ref{sat-strict}, there are vertices $a\in A$ and $c\in C$, an integer $n\ge 0$, and a path $v_0,v_1,\dots,v_n$ such that $a=v_0$, $v_n=c$ and $v_1,\dots,v_{n-1}\notin B$.

If $n=0$, then $a=c$. Thus, $A\cap C\neq \varnothing$. Hence, $X\vdash \neg(A|B|C)$ by Trivial Path axiom. Therefore, $X\nvdash A|B|C$ due to the consistency of the set $X$. 

Assume now that $n>0$. Note that $(v_i,v_{i+1})\in E$ for each $0\le i<n$, since $v_0,v_1,\dots,v_n$ is a path. Hence, $X\nvdash v_i|B|v_{i+1}$ for each $0\le i<n$ by Lemma~\ref{edge gateways strict}. Then, $\neg(v_i|B|v_{i+1})\in X$ for each $0\le i<n$ due to the maximality of set $X$. Recall that $v_1,\dots,v_{n-1}\notin B$. Hence, $X\vdash\neg(a|B|c)$ by multiple applications of Transitivity axiom because $n\ge 1$. Thus, $X\vdash\neg(A|B|c)$ by Left Monotonicity axiom. Then, $X\vdash\neg(c|B|A)$ by Symmetry axiom. Hence, $X\vdash\neg(C|B|A)$ by Left Monotonicity axiom. Hence, $X\vdash\neg(A|B|C)$ again by Symmetry axiom. Therefore, $X\nvdash A|B|C$ due to the consistency of the set $X$. 
\end{proof}

\begin{lemma}\label{confusing lemma strict}
If $a,c\notin B$ and $X\nvdash a|B|c$, then there is a path from vertex $a$ to vertex $c$ that does not contain internal vertices from the set $B$.
\end{lemma}
\begin{proof}
Since set $V$ is finite, we prove this statement by backward induction on the size of the set $B\subseteq V$. In other words, we prove the lemma by induction on the size of the set $V\setminus B$. The following three cases cover both the base and the induction steps.

\noindent{\bf Case I:} $a=c$. Consider a single-vertex path that starts and ends at vertex $a$. To finish the proof, we only need to note that this path has no internal vertices.

\noindent{\bf Case II:} $a\neq c$ and the set $\mathcal{G}(a,c)$ is empty. Thus, $(a,c)\in E$ by Definition~\ref{E-strict}. To finish this case, note that the two-vertex path $a,c$ has no internal vertices. 

\noindent{\bf Case III:}  $a\neq c$ and the set $\mathcal{G}(a,c)$ is not empty.  Take any $G\in {\mathcal G}(a,c)$ and define subsets $G_a$ and $G_c$ as follows:
\begin{eqnarray}
G_a &=& \{g\in G \;|\; X\vdash a|B,c|g\},\label{Ga}\\
G_c &=& \{g\in G \;|\; X\vdash g|a,B|c\}.\label{Gc}
\end{eqnarray}

\begin{myclaim}\label{my first claim}
$X\vdash a|B,c|G_a$ and $X\vdash G_c|a,B|c$.
\end{myclaim}
\begin{proof-of-claim}
If the set $G_a$ is empty, then $X\vdash a|B,c|G_a$ follows from the combination of Empty Set axiom and Symmetry axiom. Suppose now that the set $G_a$ is not empty. Note that $X\vdash a|B,c|g$ for each $g\in G_a$ due to definition (\ref{Ga}). Thus, $X\vdash g|B,c|a$ for each $g\in G_a$ by Symmetry axiom. Hence, $X\vdash G_a|B,c|a$ by multiple applications of Aggregation axiom because the set $G_a$ is not empty. Therefore, $X\vdash a|B,c|G_a$ by Symmetry axiom. 

Similarly, if the set $G_c$ is empty, then  $G_c|a,B|c$ is an instance of Empty Set axiom. Suppose now that the set $G_c$ is not empty. Note that $X\vdash g|a,B|c$ for each $g\in G_c$ due to definition (\ref{Gc}). Thus,  $X\vdash G_c|a,B|c$ by multiple applications of Aggregation axiom because the set $G_c$ is not empty.
\end{proof-of-claim}

The following is an instance of Insertion axiom:
$$
a|G_a,B,G_c|c \to(a|B,c|G_a\to(G_c|a,B|c\to a|B|c)).
$$
Hence,
$
X\vdash a|G_a,B,G_c|c\to a|B|c
$ by Claim~\ref{my first claim}. Recall that $X\nvdash a|B|c$ by the assumption of the lemma. Hence, $X\nvdash a|G_a, B,G_c|c$. 
Thus, by Central Monotonicity axiom, $X\nvdash a|G_a, B\cap G,G_c|c$. At the same time, $X\vdash a|G|c$ by the choice of the set $G$ and Definition~\ref{G-strict}. Recall that $G_a,G_c\subseteq G$,  hence, $G_a\cup (B\cap G)\cup G_c\subseteq G$. Then statements $X\vdash a|G|c$ and $X\nvdash a|G_a, B\cap G,G_c|c$ imply that there must exist $g\in G$ such that 

\begin{equation}\label{all about g}
    g\notin G_a,\;\;\; g\notin G_c,\;\;\; g\notin B.
\end{equation}

\begin{myclaim}\label{two a c claims}
$a\notin B\cup\{c\}$ and $c\notin B\cup\{a\}$.
\end{myclaim}
\begin{proof-of-claim}
By the assumption of the lemma, $a,c\notin B$. By the assumption of the case,  $a\neq c$.
\end{proof-of-claim}

\begin{myclaim}\label{two g claims}
$g\notin B\cup\{a\}$ and $g\notin B\cup\{c\}$.
\end{myclaim}
\begin{proof-of-claim}
Recall that $G\in \mathcal{G}(a,c)$ by the choice of the set $G$. Thus, $a, c\notin G$ by Definition~\ref{G-strict}. At the same time $g\in G$ by the choice of vertex $g$. Hence, $g\neq a$ and $g\neq c$. Finally, $g\notin B$ due to statement~(\ref{all about g}). 
\end{proof-of-claim}

\begin{myclaim}\label{a to g claim}
There is a path from vertex $a$ to vertex $g$ whose internal vertices do not belong to the set $B$.
\end{myclaim}
\begin{proof-of-claim}
Note that $g\in G$ by the choice of vertex $g$. At the same time, $g\notin G_a$ by statement~(\ref{all about g}). Thus, $X\nvdash a|B,c|g$ by statement~(\ref{Ga}). Recall that we are proving the lemma by backward induction on the size of set $B$. We now would like apply the induction hypothesis to statement $X\nvdash a|B,c|g$. To do this we must verify the following three conditions:
\begin{enumerate}
    \item $|B|<|B\cup \{c\}|$.
    \item $a\notin B\cup\{c\}$.
    \item $g\notin B\cup\{c\}$.
\end{enumerate}
The first of these conditions holds by the assumption $c\notin B$ of the lemma. The second condition holds by Claim~\ref{two a c claims}. The third condition holds by Claim~\ref{two g claims}. Therefore, by the induction hypothesis,  there is a path from vertex $a$ to vertex $g$ whose internal vertices do not belong to the set $B\cup \{c\}\supset B$. 
\end{proof-of-claim}

\begin{myclaim}\label{g to c claim}
There is a path from vertex $g$ to vertex $c$ whose internal vertices do not belong to the set $B$.
\end{myclaim}
\begin{proof-of-claim}
The proof of this claim is similar to the proof of Claim~\ref{a to g claim} except that it uses statement~(\ref{Gc}) instead of statement~(\ref{Ga}).
\end{proof-of-claim}

To finish the proof of the lemma, consider the paths whose existence stated in Claim~\ref{a to g claim} and Claim~\ref{g to c claim}. These two paths can be combined into a single path $\pi$ from vertex $a$ to vertex $c$ whose internal vertices (possibly with the exception of vertex $g$) do not belong to set $B$. Note that $g\notin B$ by statement~(\ref{all about g}). Therefore, $\pi$ is a path  from vertex $a$ to vertex $c$ whose internal vertices do not belong to set $B$.
\end{proof}

\begin{lemma}\label{about simple paths}
If there is a path from vertex $a$ to vertex $c$ that does not contain internal vertices from the set $B\setminus\{a,c\}$, then there must exist a path from vertex $a$ to vertex $c$ that does not contain internal vertices from the set $B$. 
\end{lemma}
\begin{proof}
If there is a path from vertex $a$ to vertex $c$ that does not contain internal vertices from the set $B\setminus\{a,c\}$, then there must exist a simple (without self-intersections) path $\pi$ with the same property. Any simple path from vertex $a$ to vertex $c$ does not contain vertices $a$ and $c$ as internal vertices. Therefore, path $\pi$  does not contain internal vertices from the set $B$.
\end{proof}

\begin{lemma}\label{post-confusing lemma}
If  $X\nvdash a|B|c$, then
there is a path from vertex $a$ to vertex $c$ that does not contain internal vertices from the set $B$. 
\end{lemma}
\begin{proof}
Suppose that $X\nvdash a|B|c$. Thus $X\nvdash a|B\setminus\{a,c\}|c$ by Central Monotonicity axiom. Then, by Lemma~\ref{confusing lemma strict}, there is a path from vertex $a$ to vertex $c$ that does not contain internal vertices from the set $B\setminus\{a,c\}$. Therefore, by Lemma~\ref{about simple paths}, there must exist a path from vertex $a$ to vertex $c$ that does not contain internal vertices from the set $B$. 
\end{proof}

\begin{lemma}\label{right left strict}
If $(V,E)\vDash A|B|C$, then $X\vdash A|B|C$.
\end{lemma}
\begin{proof} We consider the following four cases:

\noindent{\em Case I:} the set $A$ is empty. Then $X\vdash A|B|C$ by Empty Set axiom. 

\noindent{\em Case II:} the set $C$ is empty. Then $X\vdash C|B|A$ by Empty Set axiom. Thus, $X\vdash A|B|C$ by Symmetry axiom.  

\noindent{\em Case III:} Sets $A$ and $C$ are not empty and $X\vdash a|B|c$ for each $a\in A$ and each $c\in C$. Thus, $X\vdash A|B|c$ for each $c\in C$, by multiple applications of Aggregation axiom, due to the set $A$ not being empty. Hence, $X\vdash c|B|A$ for each $c\in C$ by Symmetry axiom. Then, $X\vdash C|B|A$ by multiple applications of Aggregation axiom, due to the set $C$ not being empty. Therefore, $X\vdash A|B|C$ by Symmetry axiom.

\noindent{\em Case IV:} There are $a\in A$ and $c\in C$ such that $X\nvdash a|B|c$. Thus, by Lemma~\ref{post-confusing lemma}, there exists a path from vertex $a$ to vertex $c$ that does not contain vertices from the set $B$. Therefore, $(V,E)\nvDash A|B|C$ by Definition~\ref{sat-strict}.
\end{proof}

\begin{lemma}\label{main induction}
$X\vdash \psi$ iff $(V,E)\vDash\psi$, for each $\psi\in\Phi(V)$.
\end{lemma}
\begin{proof}
We prove the lemma by induction on the structural complexity of a formula $\psi$. The base case follows from Lemma~\ref{left right strict} and Lemma~\ref{right left strict}. The induction step follows from Definition~\ref{sat-strict} and the maximality and the consistency of the set $X$ in the standard way. 
\end{proof}
To finish the proof of the theorem, recall that $\neg\phi\in X$. Thus, $X\nvdash\phi$ due to the consistency of the set $X$. Therefore, $(V,E)\nvDash\phi$ by Lemma~\ref{main induction}. 
\end{proof}

\section{Conclusion}\label{conclusion section}

In this article we introduced a complete axiomatic system describing properties of betweenness relation $A|B|C$ in a finite graph defined as ``every path from a vertex in the set $A$ to a vertex in the set $C$ contains at least on internal vertex from the set $B$". One can also consider a non-strict betweenness relation in which the vertex from the the set $B$ is not required to be an {\em internal} vertex of the path.




Another natural question is that of axiomatization of betweenness relation between sets of points on a line, on a plane, or, more generally, in a topological space. In the introduction section of this article we claimed without proof that although insertion principle~(\ref{symmetric insertion stronger}) is valid in an arbitrary topological space, stronger principle~(\ref{symmetric insertion strongest}) does not hold in $\mathbb{R}^2$. In fact, insertion principle (\ref{symmetric insertion strongest}) does not hold even in $\mathbb{R}$. To see the later, it is enough to consider $A=B_2=\mathbb{Q}$, $C=B_1=\mathbb{R}\setminus\mathbb{Q}$, and $I=\varnothing$. Indeed, the statements
$A|B_1,I,B_2|C$,  $A|I,C|B_1$, and  $B_2|A,I|C$ are true in this setting because between any rational number and any irrational number there is a rational number and an irrational number. Statement $A|I|C$ is false because the set $I$ is empty. In case of subsets of $\mathbb{R}^2$, the same result could be achieved by choosing $A=B_2=\mathbb{Q}\times\mathbb{R}$, $C=B_1=(\mathbb{R}\setminus\mathbb{Q})\times\mathbb{R}$, and $I=\varnothing$. The complete axiomatization of all properties of betweenness common to all topological spaces remains an open question.

\bibliographystyle{unsrt}
\bibliography{sp}

\end{document}